\documentclass{amsart}
\usepackage{graphicx}
\def\cal{\mathcal}

\input{tcilatex}

\newcommand{\EE}{{\mathbb E}}
\newcommand{\ZZ}{{\mathbb Z}}
\newcommand{\NN}{{\mathbb N}}
\newcommand{\PP}{{\mathbb P}}
\newcommand{\tPP}{\widetilde{\mathbb P}}
\newcommand{\tEE}{\widetilde{\mathbb E}}
\newcommand{\RR}{{\mathbb R}}
\newcommand{\A}{{\cal A}}

\newcommand{\D}{{\cal D}}
\newcommand{\E}{{\cal E}}

\newcommand{\G}{{\cal G}}

\newcommand{\I}{{\cal I}}

\newcommand{\R}{{\cal R}}

\newcommand{\V}{{\cal V}}
\newcommand{\X}{{\mathfrak{X}}}

\newcommand{\om}{{\omega}}

\newcommand{\oN}{{\overline{N}}}

\newcommand{\tp}{{\widetilde{p}}}
\newcommand{\tq}{{\widetilde{q}}}
\newcommand{\hN}{{\widehat{N}}}
\newcommand{\hY}{{\widehat{Y}}}
\newcommand{\hx}{{\widehat{x}}}
\newcommand{\hy}{{\widehat{y}}}
\newcommand{\hm}{{\widehat{m}}}
\newcommand{\hI}{{\widehat{I}}}
\newcommand{\wx}{{\widetilde{x}}}

\newcommand{\mb}{\mathfrak{b}}
\newcommand{\mt}{\mathfrak{t}}
\newcommand{\mh}{\mathfrak{h}}
\newcommand{\hmb}{\widehat{\mathfrak{b}}}

\newcommand{\me}{\mathfrak{e}}
\newcommand{\mE}{\mathfrak{E}}
\newcommand{\mx}{\mathfrak{w}}
\newcommand{\mC}{\mathfrak{C}}
\newcommand{\mP}{\mathfrak{P}}
\newcommand{\mD}{\mathfrak{D}}

\input tcilatex
\DeclareUnicodeCharacter{207A}{\ensuremath{^+}}  

\begin{document}
\newtheorem{property}{Property}

\title[Notes]{Level numbers preserving transformations on excursions
defined by random walks with state dependent jump laws}

\author{Thierry Huillet$^{1}$, Servet Martinez$^{2}$}
\address{$^{1}$Laboratoire de Physique Th\'{e}orique et Mod\'{e}lisation \\
CNRS-UMR 8089 et Universit\'{e} de Cergy-Pontoise \\  
$2$ Avenue Adolphe Chauvin \\
95302, Cergy-Pontoise, FRANCE \\
$^{2}$ Departamento de Ingenier\'{i}a Matem\'{a}tica \\
Centro Modelamiento Matem\'{a}tico\\
UMI 2807, UCHILE-CNRS \\
Casilla 170-3 Correo 3, Santiago, CHILE.\\
E-mail: Thierry.Huillet@u-cergy.fr and smartine@dim.uchile.cl}

\begin{abstract}
We study  the number of individuals per level defined by excursions of random 
walks with state dependent jump law. These level numbers
determine the probability of the excursion, and the set of 
transformations preserving the level numbers is, generically, the set 
of transformation that preserve the probability law of excursions.
We compute the number of excursions having a fixed level numbers  
and we show that the class of shifts of excursions generate all 
the excursions having a fixed level numbers.  We study the behavior
of the level numbers under the Vervaat transform and   
the  Doob transform.
\end{abstract}

\maketitle

\textbf{Running title}: Random walks and excursions.\newline



\medskip

\section{Introduction}

The framework of this study is given by Section $6$ of \cite{Harris} which
also constitutes our main reference. 
In his article of first passage times on random walks whose jump law 
is state dependent, Harris associated trees to excursions and show that 
for the homogeneous subcritical random walk this is the the tree of a 
linear fractional Galton 
Watson process. In our work the jump law of the random walk is also state 
depending and the homogeneous case serves to illustrate the results.

\medskip   

We study classes of transformations on the excursions and focus 
on those preserving the (occupation) level numbers of the associated trees.  
As examples of these classes are the reversed transformation and 
the shift of bridges.

\medskip

One of the motivation for this study comes from the following observation. 
The probability of  an excursion only depends on the level numbers. Moreover,
except for jumps laws satisfying some integer relation,  the probability of the
excursion also determines the level numbers. 
So the probability preserving transformation on excursions 
are, in a generic way, those preserving the level numbers.

\medskip

In Proposition \ref{propcuenta} we compute 
the number of all the excursions having a fixed level numbers, and 
in Proposition \ref{mainprop} we show 
that the family of shifts of excursions allow to recover all the excursions 
having a fixed level numbers. Its proof gives an algorithm to get all of them.  

\medskip

For the excursions attaining their height at a unique point,  
we study the Vervaat transform, determine the change of
its level numbers and compute the probability of its domain of definition. 
In the last section we study the Doob transform of the random
walk ensuring to have finite excursions a.s.  In Proposition \ref{propdoob} 
it is shown that the law of the excursions of the Doob transformed 
random walk retrieves the law of
the excursions of the original random walk conditioned to be finite.
In the homogeneous case this gives the law of the level counting process 
defined by the Vervaat transform.

\medskip

There is an important and huge literature devoted to excursions of random 
walks and finite trees, we mention \cite{ch} and all the references therein. 
The branching property on the Galton Watson tree processes 
has been shown by several authors, at this respect see 
Proposition of Section 3 in \cite{Neveu} and Theorem 2.7 in \cite{ch}.
For some counting problems of
trees we refer to \cite{Pit}, in particular to compute the number of 
trees having a fixed class of children of the nodes. 
The asymptotic properties of several counting problems of finite rooted trees  
defined by the symmetric random walk, are described in \cite{mar}.

\medskip

An important part of these notes is devoted to fix the notation on
excursions, individuals, trees, level numbers and in the excursion random 
variable.

\medskip

\section{Excursions of a random walk}

By $\ZZ_+=\{0,1,2,...\}$ we denote the nonnegative
integers and by $\NN=\{1,2,...\}$ the positive integers. By $|A|$ 
we mean the cardinality of the set $A$.

\subsection{Excursions}

A positive excursion is a finite sequence of points 
$x=x[0,\theta]=(x_n: n=0,...,\theta)$ in $\ZZ_+$ that satisfies 
$$
x_0=0=x_\theta, \; x_n\!>\!0 \hbox{ for } n\! \in \!(0,\theta),
\hbox{ with jumps  }y_n\!=\!x_{n+1}\!-\!x_n\! \in \! \{1,-1\},  
n\! \in \! [0,\theta).
$$
Its length is $\theta(x)=\theta$, which is 
an even number because $\theta/2$ is the number of
$1$ jumps which is equal to the number of $-1$ jumps. 
A negative excursion is defined analogously except that $x_n<0$
for $n\in (0,\theta)$. 

\medskip

Let $\X$ be the countable set of all excursions (in \cite{Pit} 
they are called lattice excursions), so 
$\X=\X^+\cup \X^-$ where $\X^+$ (respectively $\X^-$) is the class 
of all positive (respectively negative) excursions. The sign change
$x\to -x$ defines one-to-one mappings $\X^{+}\to \X^{-}$ and 
$\X^{-}\to \X^{+}$, that preserve the length of the excursions.
The number of excursions with a fixed length 
$|\{x\in \X: \theta(x)=\theta\}|$, is given by the Catalan numbers 
(see Section $6$ in \cite{Pit}). 

\medskip

The height $x\in \X$, denoted $H(x)$, is defined by 
$$
H(x)=\max x[0,\theta] \hbox{ for } x\in \X^+, \; 
H(x)=\min x[0,\theta]  \hbox{ for } x\in \X^-.
$$
(For negative excursions we prefer to call it height instead of depth, 
as it is the usual 
name). We have  $H(x)\le \theta/2$ and $H(-x)=-H(x)$.

\medskip

\subsection{Excursion random variable defined by a random walk}

Let $X=(X_n: n\ge 0)$ be a random walk on $\ZZ$ 
with independent jumps $Y_n=X_{n+1}-X_{n}$ taking values in $\{-1,1\}$. 
The law $\PP$ of the random walk is defined by the transition probabilities 
$p_k=p(k,k+1)\in (0,1)$ for a jump  from $k$ to $k+1$ and 
$q_k=p(k,k-1)=1-p_k$ from $k$ to $k-1$, for $k\in \ZZ$. 
The sequence of jumps
$Y=(Y_n: n\ge 0)$ satisfies $\PP(Y_n=1)=p_{X_n}$
and $\PP(Y_n=-1)=q_{X_n}$. In the {\it homogeneous} case 
the sequence of jumps is Bernoulli with 
$p_k=p$, $q_k=q=1-p$ for $k\in \ZZ$.

\medskip

Let $\PP_i$ be the law
of the walk starting from $i\in \ZZ$ and set $\PP=\PP_0$ when
the walk starts from $0$. When starting from $X_0=0$,   
the first return time to $0$ is denoted by $\Theta=\inf\{n>0: X_n=0\}$.         
The excursion random variable  
is defined only in the set $\{\Theta<\infty\}$ and it is given by 
$$
X^\Theta=X[0,\Theta]=(X_n: n\in [0,\Theta]).
$$
So, $X^\Theta$ takes values on $\X$ and it
inherits all the notation of excursions, thus 
its sign reflection is $-X^\Theta$, it has 
length $\Theta$ and height $H(X^\Theta)$. 
Even if $X_0=0$ when defining $X^\Theta$, 
by an abuse of notation we set $\PP_1$ to 
mean that the excursion is positive and 
starts from the state $1$. Similarly for negative excursions.

\medskip

The condition $\PP_1(\Theta < \infty)=1$ 
is equivalent to $\alpha_{\infty}=\infty$ with 
$\alpha_{\infty}=\lim\limits_{n\to \infty} \alpha_n$, where
\begin{equation}
\label{eqG1}
\alpha_0=1, \alpha_1=1 \hbox{ and }
\alpha_n=\sum_{i=1}^n\prod_{k=1}^{i-1} (q_k/p_k) 
\hbox{ for } n\ge 2.
\end{equation}
Let us consider the probabilities
$\beta_i=\PP(\Theta < \infty | X_k=i, k<\Theta)$  for  $i\ge 1$,
 where $k$ is any positive number that satisfies $k-i\in 2\ZZ_+$. 
We have:
\begin{equation}
\label{eqG2}
\forall i\ge 1, \; \beta_i=
\begin{cases}
1 & \hbox{ if } \alpha_{\infty}=\infty\\
1-\frac{\alpha_i}{\alpha_{\infty}}
=\frac{\sum_{j=i}^\infty\prod_{k=1}^{i-1} (q_k/p_k)}
{\sum_{j=0}^\infty\prod_{k=1}^{i-1}(q_k/p_k)} & \hbox{ if } \alpha_\infty<1. 
\end{cases}
\end{equation}  
All these computations are given in Theorem 2b in \cite{Harris}. 
For the negative excursions these quantities are defined similarly. So, 
as in (\ref{eqG1}) and (\ref{eqG2}),   
$\alpha_{-\infty}=\lim\limits_{n\to \infty}a_{-n}$ with 
$a_{-n}=\sum_{i=1}^n\prod_{k=1}^{i-1}(p_{-k}/q_{-k})$.
Also $\beta_{-i}=\PP(\Theta < \infty | X_k=-i, k<\Theta)$, with 
$k+i\in 2\ZZ_+$, satisfies:
\begin{equation}
\label{eqG2x}
\forall i\ge 1, \;  \beta_{-i}=1 \hbox{ if } \alpha_{-\infty}=\infty,  \hbox{ or }
\beta_{-i}=1-\frac{\alpha_{-i}}{\alpha_{\infty}}  
\hbox{ if } \alpha_{-\infty}<\infty.
\end{equation}

We have, 
\begin{eqnarray}
\nonumber
\PP(\Theta<\infty)=\PP(X^\Theta\in \X^+)+\PP(X^\Theta\in \X ^-) \hbox{ and }\\
\label{descx}
\PP(X^\Theta\in \X^+)=p_0 \beta_1, \; \PP(X^\Theta\in \X^-)=q_0\beta_{-1}. 
\end{eqnarray}
Then,
$$
\PP(X^\Theta\in A | X^\Theta\in \X^+)=(p_0 \beta_1)^{-1} \, \PP(X^\Theta\in A)
\hbox{ for}  A\subseteq \X^+.
$$
Hence, the law induced by the random walk on the class of positive excursions 
$\X^+$ is given by 
\begin{equation}
\label{prob1G}
P(X^\Theta=x | X^\Theta\in \X^+ )
=(p_0 \beta_1)^{-1} \, \prod_{n=0}^{\theta(x)-1}p(x_n, x_{n+1}) ,\; 
x\in \X^+.
\end{equation} 

In the homogeneous case the condition $\alpha_{\infty}=+\infty$ is 
equivalent to $p\le q$, so $\beta_i=1$ for all $i\ge 1$ and,
\begin{equation}
\label{eqG3}
\hbox{ if } p>q \hbox{ then }   \alpha_\infty=\frac{p}{p-q} \hbox{ and }
\beta_i=\left(\frac{q}{p}\right)^i, \hbox{ for } i\ge 1.
\end{equation}
Then, in this case, the probability of an excursions only depends on 
its length. For $x\in \X^+$ 
it  is $P(X^\Theta=x^\theta | X^\Theta\in \X^+)=(pq)^{\theta/2}/p$ if $p\le q$ 
and $P(X^\Theta=x^\theta | X^\Theta\in \X^+)=(pq)^{\theta/2}/q$ if $p>q$.

\medskip

Similar expressions are obtained for negative excursions.

\medskip

In the next two Sections \ref{sss0} and \ref{sss1}, we fix the notation 
of concepts introduced in \cite{Harris}.

\medskip

\subsection{Individuals: birth and death times, level}
\label{sss0}

An individual $I$ of a $x\in \X^{+}$ is characterized by a 
triple $(\mb(I),\mh(I),\mt(I))$. If there is no possible confusion we write it by 
$(\mb,\mh,\mt)$. These quantities satisfy $0\le \mb<\mt\le \theta$, 
$\mh\ge 0$ and
$$
x_{\mb}=\mh,\; y_{\mb}=1, \; \mt=\inf\{n>\mb: x_n=\mh\}. 
$$
So, $x_{\mb+1}=\mh+1, \; y_{\mt-1}=-1$.
One says that the individual $I$ is 
born at time $\mb$ and at level $\mh$, and dies at time 
$\mt$. Notice that no individual can be born 
at time $\mt-1$. For a fixed $x\in \X^+$, $I$ 
is characterized only by $\mb$, because $\mh=x_\mb$ and
$\mt$ is a function of $\mb$ and $\mh$. 
But also $I$ is characterized by its time of death $\mt$, 
because $\mh=x_\mt$ and since in $(\mb,\mt)$ no individual can be born 
at level $\mh$ one has $\mb=\max\{0\le k<n: x_k=\mh \}$. 
Then, for $n\in [0,\theta)$,  $y_n=1$ marks the birth of an individual 
at $n$ and $y_{n}=-1$ marks the death of an individual at $n+1$.
 Hence, with the 
set of times or birth or the set of times of death of individuals, one 
determines  the set of jumps and so the excursion.

\medskip

Let $\I(x)$ be the set of individuals in $x$. Its cardinal number satisfies
$|\I(x)|=\theta(x)/2$ because it is the number of $1$ jumps.

\medskip

One has $x_{\mb}=\mh=x_{\mt}$ and $x(\mb,\mt)> \mh$, so
when we shift the levels by $-\mh$, the trajectory   
$x[\mb,\mt]-\mh=(x_n-\mh: n\in [\mb,\mt])$ is an excursion. 
We denote $x(I)=x[\mb,\mt]$ and call it  the excursion of $I$ in $x$.
The life length of $I$, denoted by $\theta(I)$, is the length of the 
excursion $x(I)$ which is $\theta(x(I))=\mt-\mb$.
There is a unique individual $I_0$
born at time $\mb=0$ at level $\mh=0$ and dying at time $\mt=\theta$.
So, $x(I_0)=x$ and $\theta(I_0)=\theta(x)$.

\medskip

When needed we write $I(x)$, $\mb(I(x))$, $\mh(I(x))$
and $\mt(I(x))$ to express the dependence of individuals on 
the excursion $x$.

\medskip

Individuals can be defined also for a negative excursion $x$.
In this case, an individual $I$ satisfies 
$0\le \mb<\mt\le \theta(x)$ but now $\mh<0$, and 
$x[\mb,\mt]-\mh=(x_n-\mh: n\in [\mb,\mt])$ is a 
negative excursion.

\medskip

\subsection{Tree and order of an excursion}
\label{sss1}

Let $x\in \X^{+}$ be a positive excursion and
$I\in \I(x)$ be given by $(\mb,\mh,\mt)$. The individuals in the set
$$
\mC(I)=\{J\in \I(x): \mh(J)=h+1, \mb(J)\in [\mb+1,\mt-1] \},
$$
are called the children of $I$, and $I$ is said to be the parent of them,
we put $I=\mP(J)$ for all $J\in \mC(I)$. $I$ is a 
leaf if $\mC(I)=\emptyset$.
An individual born in $[\mb+1,\mt-1]$ is called a successor of $I$. 
For all $0\le h<\mh(I)$ there is a unique predecessor of $I$ at level $h$.
This defines a (finite) tree $T(x)$ rooted by $I_0$ and set of sites $\I(x)$.
The tree hanging from $I$ is noted $T(I)$, it is rooted by $I$ and 
its set of nodes is constituted by $I$ and all its successors.

 \medskip

The individuals born at level $\mh+1$ in $[\mb+1,\mt-1]$ are born 
at times $\{n\in [\mb+1,\mt-1]: x_n=\mh+1\}$, and they 
die before $\mt$ because the numbers of jumps $1$ and $-1$
are the same in $[\mb+1,\mt-1]$. 
The equality $|\{n\in [\mb+1,\mt-1]: x_n=\mh+1\}|=k$ expresses
that $k-1$ individuals are born in $[\mb+1,\mt-1]$ 
at level $\mh+1$ in $x$, because the passage to $\mh+1$ 
at time $\mt-1$, satisfies $y_{\mt-1}=-1$. Notice that 
$k=1$ if and only if $\mt=\mb+2$.

\medskip

The excursion of the individual $I$, $x(I)=x[\mb,\mt]$, with 
$x_{\mb}=\mh=x_{\mt}$, has the following structure: 
\begin{eqnarray*}
&{}& [\mb\!+\!1,\mt\!-\!1]=\bigcup_{i=1}^k [a_i,b_i] \hbox{ with } 
x(a_i)\!=\! \mh \!+\!1\!=\!x(b_i),\,  x(a_i,b_i)\!>\! \mh \!+\!1, i=1,..,k; \\
&{}& b_i=a_{i+1}, i=1,..,k-1; \;  x[a_i,b_i]=x(J_i), \, J_i\in \mC(I).
\end{eqnarray*}
That is, the class of excursions of all the children in $\mC(I)$ are contiguous 
in the excursion $x$. The individual $I$ is born at the site just before 
this class of excursions and dies just after this class.

\medskip

In a reciprocal way. Let $h\ge 0$. If $x_a=h+1=x_b$ and $x(a,b)\ge h+1$, then 
$x[a,b]=(x(J_1),x(J_2),..,x(J_l))$ where 
$J_1,..,J_k\in \mC(I)$ are children 
of an individual $I$ born at level $h$ and at time 
$\mb=\max\{0\le n<a: x_n=h \}$.

\medskip

Let $\I(x)=\{I_i: i=0,..,|\I|-1\}$ be endowed with the following order 
that we call level order. One puts   
\begin{equation}
\label{order}
i\!<\! j \hbox{  if }  \mh(I_i)\!<\! \mh(I_j) \hbox{ or if } \mh(I_i)\!=\!\mh(I_j)
\hbox{ and }\mb(I_i)\!<\! \mb(I_j).
\end{equation} 
Since the time of birth of the successors of an individual is bigger 
than the one the individual, then the level order and the
parent relation determine the order of the times of birth of individuals.
Finally, since the excursions of the individuals with the 
same parent are contiguous, we get
\begin{equation}
\label{presor}
\hbox{For } \mh(I_i)=\mh(I_j), I_u\in \mC(I_i), I_v\in \mC(I_j) \hbox{ we have } 
[i<j\, \Rightarrow  \, u<v].
\end{equation}
That is, the order between children of individuals at the same level   
follows the order of their parents.

\medskip

\subsection{Representation of excursions and trees}

The excursion $x\in \X^{+}$ is represented by the sequence
$$
\mx_x=(\mx_x(j): j\in [0,\theta(x)) \hbox{ with } 
\mx_x(\mb(I_i))=i=\mx_x(\mt(I_i)-1), \; I_i\in \I(x),
$$
that is well defined because for all $n\in [0,\theta(x))$ 
one has that an individual is born at $n$ or dies at $n+1$.   
Every $i\in \{0,..,,\theta(x)-1\}$ appears twice in $\mx_x$, at the time
of birth of $I_i$ and at one unit before its time of death, and
$\mx_x$ starts and finishes at the root $I_0$, $\mx_x(0)=I_0=\mx(\theta(x)-1)$. 
Then, this representation contains the same 
representation as $x$ because the jumps of $x$ are given by
\begin{equation}
\label{repseq2}
y_j=
\begin{cases}
1 & \hbox{ if } \forall 0\le k<j: \;  \mx_x(k)\neq \mx_x(j);\\
-1 & \hbox{ if } \exists 0\le k<j: \; \mx_x(k)=\mx_x(j).
\end{cases}
\end{equation}

Now, let $T$ be a tree with root $a_0$ and set of nodes 
$\A(T)=\{a_0,..,a_\ell\}$.  We denote by 
$\mC(a_i)$ the set of children of $a_i$ and 
by $\mP(a_i)$ the parent of $a_i\neq a_0$. 
The root $a_0$ is the unique node at level $0$
of $T$ and the nodes at level $h+1$ of $T$ are the children of nodes at 
level $h$. If for every node $a\in \A(T)$ that is not a leaf, 
$\mC(a)$  is totally ordered  by some relation $\preceq_a$ we say that 
the rooted tree is ordered.  
When this happens one can refer to the first or to 
the last child, or to the child following 
some child. The tree $T(x)$ defined by an excursion $x$ is a rooted ordered 
tree because the children of every individual are ordered 
by the time of their birth.

\medskip

Let us see that if $T$ is a rooted ordered tree, then it determines a unique 
excursion $x$ with $T(x)=T$.  
We first define a sequence $\mx^T=(\mx^T(j): j=0,...,\ell-1\}$
 with values in  $\A(T)$ as follows. We take $\mx^T(0)=a_0$. 
If the root is a leaf  we put $\mx^T(1)=a_0$ and the construction finishes. 
If not we put $\mx^T(1)=a$ the first child of $a_0$.  Now 
we make an inductive construction for which it holds: 
\begin{equation}
\label{claimx}
\mx^T(i+1) \hbox{ can only be } \mx^T(i),  P(\mx^T(i))
\hbox{  or belong to } \mC(\mx^T(i)).
\end{equation}
Let us  assume we have constructed $(\mx^T(j): j\le i)$, set $a=\mx^T(i)$. 

\smallskip

$\bullet$ Assume $\mx^T(i-1)=P(a)$. If $a$ is a leaf then 
we put $\mx^T(i+1)=\mx^T(i)$. 
If $a$ is not a leaf then $\mx^T(i+1)\in \mC(\mx^T(i))$ is the first child 
of $a$;

\smallskip

$\bullet$ Assume $\mx^T(i-1)=a$. In this case $a$ is a leaf. We set 
$\mx^T(i+1)=P(\mx^T(i))$ and we stop if $P(a)=a_0$;

\smallskip

$\bullet$ Assume $\mx^T(i-1)\in \mC(a)$.  If $\mx^T(i-1)$ is the last children 
of $a$, we set $\mx^T(i+1)=P(\mx^T(i))$ and we stop if $P(a)=a_0$. 
If not, we put $\mx^T(i+1)\in \mC(\mx^T(i))$ the child of 
$a$ that follows $\mx^T(i-1)$.
 
\medskip

This construction satisfies (\ref{claimx}). The sequence 
$\mx^T$ contains each node exactly two times and it starts and finishes at 
the root $a_0$.  From  (\ref{repseq2}), we can associate to $\mx^T$ an excursion 
$x^T$ having length $\theta(x)=2\ell$. This is the excursion of the contour process 
of a tree $T$, see  Section 3.3 in \cite{ch}. 

\medskip

Let $T$ and $T'$ be rooted trees with roots $a_0(T)$ and $a_0(T')$ respectively. 
They are said to be equivalent, we denote
it by $T\equiv T'$, if there exists a one-to-one mapping $\xi: \A(T)\to \A(T')$ such that 
$\xi(a_0(T))=a_0(T') \hbox{ and } \forall a,b\in \A(T):\;
\big[b\in \mC(a) \Leftrightarrow \xi(b)\in \mC(\xi(b))\big]$. 
This is equivalent to have a one-to-one mapping 
$\xi:\A(T)\to \A(T')$ that preserves the levels of $T$ and $T'$  and 
for all $a\in \A(T)$ the number of children of $a$ and $\xi(a)$  are the same.  

\medskip

For $x\in \X^+$ it is satisfied, 
\begin{equation}
\label{propx1}
T\equiv T(x) \Leftrightarrow  
\forall a\in \A(T) \; \exists \hbox{ a total order on } \mC(a) \hbox{ such that } x^T=x.
\end{equation}
It suffices to show $\Rightarrow$. The total order is defined recursively.  
We put $a_0(T)=\xi(I_0)$. Since $T\equiv T(x)$ there is a one-to-one 
assignment $\xi_{I_0}: \mC(I_0)\to \mC(a_0(T))$,   
such that $T(I)$ and $T(\xi_{I_0}(I))$ are equivalent. 
 We enumerate  $\mC(a_0(T))=\{\xi_{I_0}(I): I\in \mC(I_0)\}$ by the time of born of 
$I\in \mC(I_0)$. Now, the induction is made for all $I\in \mC(I_0)$ because
we can state a one-to-one assignment $\xi_{I}:\mC(I)\to \mC(\xi_{I_0}(I))$ such that
$T(J)$ and $T(\xi_{I}(J))$ are equivalent and we enumerate 
$\mC(\xi_{I_0}(I))=\{\xi_{I_0}(J): J\in \mC(I)\}$ by the time of born of 
$J\in \mC(I)$.  
Thus, we get $\mx^T=\mx_x$, so (\ref{repseq2}) gives $x^T=x$.   

\medskip

\begin{remark}
\label{Pitxx}
When one takes the collection of number of children of the 
individuals,  $\{|\mC|\}(x)=\{|\mC(I)|: I\in \I(x)\}$, then 
one can  compute the class of trees 
$T\in \{T\}\{|\mC|\}(x)$ that satisfy  $\{|\mC(a)|: a\in \A(T)\}=\{|\mC|\}(x)$, 
see \cite{Pit}. 
\end{remark}

\subsection{Level numbers of an excursion}

Let $\I_h(x)$ be the set of individuals born at level $h$ for $x\in \X$. 
For $x\in \X^+$, the sequence 
$$
N(x)=(N_h(x)=|\I_h(x)|: h\ge 0)
$$ 
of the number of individuals in $x$, ordered by the level at which they are born, 
is called the level numbers of $x$. We have $N_0(x)=1$ and
the total number of individuals $\oN(x)=\sum_{h\ge 0} N_h(x)$ 
satisfies $\oN(x)= |\I(x)|= \theta(x)/2$. 
The height $H(x)$ is the time of extinction of $N(x)$ because
$N_{H(x)}(x)=0$and $N_{H(x)-1}(x)>0$. So $H(x)$ is also called 
the height of $N(x)$ and denoted by $H(N(x))$.

\medskip
  
For a negative excursion $x$ one associates the  the sequence 
of level numbers $N(x)=(N_h(x)=|\I_h(x)|: h\le 0)$ 
in a similar way and the height $H(x)$ is also the time of extinction
of $N(x)$. 

\medskip

In the set $\I(x)=\{I_i: i=0,..,|\I|-1\}$ endowed with the level order 
(see (\ref{order})), the index $i$ of $I_i$ satisfies 
$i=(\sum_{k=0}^{h(I_{i-1})} N_k(x))+r_i-1$, 
where $r_i$ is the rank of time of birth of $I_i$ at level $\mh(I_i)$.

\medskip

\subsection{Law on excursions and level numbers}

Let us see that the level numbers of an excursion determines its 
$\PP$ probability measure. 

\begin{proposition}
\label{propG0}
For $x\in \X ^+$, the probability 
$\PP(X^\Theta=x | X^\Theta\in \X^+)$ only depends on $N(x)$. Or, equivalently, if 
$x, x'\in \X^+$ satisfy $N(x)=N(x')$, then 
$\PP(X^\Theta=x | X^\Theta\in \X^+ )=\PP(X^\Theta=x' | X^\Theta\in \X^+ )$. 

\smallskip

In a reciprocal way, assume the following hypothesis holds:  the set 
$\{\log(p_h q_{h+1}): h\ge 1\}$ satisfies
\begin{equation}
\label{eqarit}
\not\exists 
K\subset \NN,  K \hbox{ finite}, K\neq \emptyset: \, 
\sum_{h\in K} a_h \log (p_h q_{h+1})=0 \hbox { with } a_h\in
\ZZ\setminus \{0\}, h\in K. 
\end{equation}
Then,
if $\PP(X^\Theta=x | X^\Theta\in \X^+ )=\PP(X^\Theta=x' | X^\Theta\in \X^+ )$ 
we necessarily have
$N(x)=N(x')$. So, in this case the probability determines the level numbers.
\end{proposition}

\begin{proof}
Since all the individuals born at level $h$ 
also die at this level, the probability of an excursion 
$x$ given by (\ref{prob1G}), is also expressed as
\begin{equation}
\label{prob2G}
\PP(X^\Theta=x | X^\Theta\in \X^+ )
=(p_0 \beta_1)^{-1} \!\! 
\prod_{h=0}^{H(x)-1} (p_h q_{h+1})^{N_h(x)},
\end{equation}
which gives the first part.

\medskip

For the reciprocal, let $x, x'\in \X^+$. 
If $\PP(X^\Theta=x | X^\Theta\in \X^+ )=
\PP(X^\Theta=x' | X^\Theta\in \X^+ )$, 
from formula (\ref{prob2G}) one gets, 
$$
\prod_{h=0}^{\theta(x)-1} (p_h q_{h+1})^{N_h(x)}=
\prod_{h=0}^{\theta(x')-1} (p_h q_{h+1})^{N_h(x')}.
$$
Now we erase from both sides the common terms of the 
type $p_l q_{l+1}$, in particular we erase $p_0 q_1$.
There will remain a term at the left or at the right hand side 
only if and only if
$N(x)\neq N(x')$. When this last case takes place we are left with an 
equality
$$
\prod_{h\in J}(p_h q_{h+1})^{b_h}=\prod_{h\in J'}(p_h q_{h+1})^{c_h}\,,
$$ 
with $b_h, c_h\in \NN$, $J$ and $J'$ disjoint finite subsets of $\NN$, 
and at least one of them being non nonempty. 
By taking $\log$ the hypothesis (\ref{eqarit})  is contradicted.
\end{proof}

\medskip

Since, only for a countable set of positive real numbers there 
exists an arithmetic integer condition (\ref{eqarit}), 
the generic case is that  the probability of the excursions 
are equal if and only if the level numbers are equal. 

\medskip

In the homogeneous case one has 
$\PP(X^\Theta=x | X^\Theta\in \X^+ )
=\PP(X^\Theta=x'| X^\Theta\in \X^+ )$
if and only if $\theta(x)=\theta(x')$, that is the equality of probabilities
only gives $\oN(x)=\oN(x')$.

\medskip

For negative excursions a similar statement as the one of Proposition \ref{propG0}
can be written. The unique difference is that for a negative excursion $x\in \X^-$ 
one has  $\PP(X^\Theta=x | X^\Theta\in \X^-)
=(q_0 \beta_{-1})^{-1} \prod_{h=H(x)+1}^{0} (q_h p_{h+1})^{N_h(x)}$.

\medskip

For the excursion random variable $X^\Theta$, the random element  
$N(X^\Theta)=(N_h(X^\Theta): h\ge 0)$ is its
associated level counting process.  
For an individual $I$ given by $(\mb, \mh, \mt)$, the probability that 
$k$ individuals are born in $[\mb+1,\mt-1]$ 
at level $\mh+1$, that is of having $|\mC(I)|=k-1$, is $p_{\mh}^{k-1}q_{\mh+1}$.

\medskip

In the homogeneous 
case  one has $p_h=p$, $q_h=q$, so the number of children 
of any individual is distributed as $G\sim \,$ Geometric$(q)$-1, so with 
$\PP(G=k)=p^{k}q$ for $k\ge 0$ and generating  function  
$\EE(z^G)=q/(1-p z)$.  So, $N(X^\Theta)$ 
is a linear fractional Galton Watson branching process of parameter $p$.
Many of its properties have been developed in Section I.4 of \cite{AN}. 
In \cite{Kleb}  it is studied Doob transforms on 
the Markov chain of the linear fractional branching process 
or in its generating function. 
In Section \ref{ssdoob} we study the Doob transform  
on the state dependent random walk to get the law  
of $X^\Theta$ and $N(X^\Theta)$ conditioned to $\Theta$ be finite.

\medskip

\subsection{Numbers of excursions with the same level numbers}

Let $N=(N_h: h\ge 0)$ be a (sequence of) level 
numbers. This means $N_h\in \ZZ_+$ for $h\ge 0$, 
$N(0)=1$, if $N_H=0$ then $N_{H+k}=0$ for all $k\ge 0$, 
and the total number of individuals $\oN=\sum_{h\ge 0}N_h$ is finite. 
We call $H(N)=\min\{h\ge 1: N_h=0\}$ the height or the 
time of extinction of $N$. Let $\theta=2\oN$. 
Consider the set of positive excursion with level numbers $N$, 
$$
\X^{+}(N)=\{x\in X^{+}: N(x)=N\}.
$$
Every $x\in \X^+(N)$ satisfies $\theta(x)=\theta$.
Let us check that  $\X^{+}(N)\neq \emptyset$. Let $H=H(N)$ 
and define the excursion $x$ given by the sequence of jumps 
$(y_n: n\in [0,\theta-1])$ with
$y_n=1$  if $n\in [0,H-1]$, and after  
$N(h)-1$ individuals are born at each level $h=H-1,H-2,...,1$, 
and they die immediately after their birth times. This is:
\begin{eqnarray*}
y_{H+2(N_{H-1}+...+N_{h+1})+2n}=-1 &{}& \hbox{ if } 
n\in [0,N_h-1],\\
y_{H+2(N_{H-1}+...+N_{h+1})+2n+1}=1 &{}& \hbox{ if } n\in [0,N_h-2].
\end{eqnarray*}
Then, there are born $N_h$ individuals at level $h=1...,H-1$, 
there is $1$ individual born at level $h=0$, and no individual 
is born at level $H$. So, $N(x)=N$.

\medskip

For making computations on the distributions of excursions, 
it is useful to compute the numbers of the excursions having 
the same level numbers, this is the cardinality of $\X^+(N)$, 
because combined with (\ref{prob2G}), it gives the 
probability of the class of excursions having the same level numbers.

\medskip

\begin{proposition}
\label{propcuenta}
Let $N=(N_h: h\ge 0)$ be a (sequence of) level 
numbers with $N(0)=1$, finite
total number of individuals $\oN$ and height $H(N)$. 
Then, the cardinality of $\X^{+}(N)=\{x\in X^{+}: N(x)=N\}$ is
$$
|\X^+(N)|=\prod_{h=1}^{H(N)-2}\binom{N_{h+1}+N_h-1}{N_h-1}\;.
$$ 
\end{proposition}

\begin{proof}
Let $M,k\ge 1$. Let
$\D(M,k)=|(s_1,..,s_k)\in \ZZ_0^k: \sum_{j=1}^k s_i=M\}$ be the set of 
(additive) decompositions of $M$ into $k$ non-negative integers. Then 
$$
|\D(M,k)|=\binom{M+k-1}{k-1}.
$$
To check it, let $P(M\!+\!k\!-\!1,k\!-\!1)$ be the class of subsets of 
$\{1,...,M\!+\!k\!-\!1\}$ having $k\!-\!1$ elements.
The elements $\{l_1,..,l_{k-1}\}\in P(M\!+\!k\!-\!1,k\!-\!1)$ are written in 
an increasing form, $l_1< ...< l_{k-1}$. We put $l_0=0$, $l_{k}=M\!+\!k$.
Then, the mapping $P(M\!+\!k\!-\!1,k\!-\!1)\to \D(M,k)$, 
$\{l_1,..,l_{k-1}\}\to (s_1,...,s_k)$ with $s_j=l_{j}-l_{j-1}-1$, $j=1,...,k$, 
is a bijection.   

\medskip

Now take $\mD(N)=\prod_{h=1}^{H(N)-2} D(N_{h+1},N_h)$. This 
set has cardinality 
$$
|\mD(N)|=\prod_{h=1}^{H(N)-2}\binom{N_{h+1}+N_h-1}{N_h-1}. 
$$
To finish the proof let us show that there is a bijection 
$\eta: \mD(N)\to \X^+(N)$. Let 
$$
s=(s^h: h=1,...,H(N)-2)\in \mD(N) \hbox{ with }  
(s^h_1,..,s^h_{N_h})\in D(N_{h+1},N_h),
$$ 

First, we associate to $s$ a tree $T^s$, that has a root $a_0$ (level $0$). 
The root has $N_1$ children $\mC(a_0)=\{a^1_1,..,a^1_{N_1}\}$, ordered 
by the subindex (these are the elements at level $1$). Each
$a^1_{i_1}$ has $s^1_{i_1}$ children, for $i_1=1,..,N_1$. 
The children of $a^1_{i_1}$ 
are noted $\mC(a^1_{i_1})=\{a^2_{i_1,i_2}: i_2=1,...,s^1_{i_1}\}$, 
so $i_2$ indicates the 
rank of $a^2_{i_1,i_2}$ in $\mC(a^1_{i_1})$. 
Since $\sum_{i_1=1}^{N_i} s^1_{i_1}=N_2$, 
we can enumerate the $s^1_{i_1}$ children of $a^1_{i_1}$ as 
$$
\mC(a^1_{i_1})=\{a^2_i: i=\sum_{j_1=1}^{i_1-1}s^1_{j_1}+i_2, i_2=1,...,s^1_{i_1}\}.
$$
In this way there are $N_2$ nodes (elements of level 
$2$ of $T^s$) that are children of the $N_1$ children of $a_0$, that 
are enumerated by    
$\{a^2_i: i=1,...,N_2\}$  with  $i=i(i_1,i_2)=\sum_{j_1=1}^{i_1-1}s^1_{j_1}+i_2$ 
for some $i_2=1,..,s^1_{i_1}$. 

\medskip

By a recursive argument, the elements of the level $h\le H(N)-2$ of 
$T^s$ can be enumerated as $\{a^h_{i}: i=1,...,N_h\}$ where the rank $i$
depend on a $h-$ tuple of indexes $i=i(i_1, i_2,..,i_h)$ meaning that
$a^h_i\in \mC(a^{h-1}_{j})$ with $j=j(i_1,...,i_{h-1})$ and $i_h$ is the rank 
of $a^h_i$ in the set of children $\mC(a^{h-1}_{j})$.  
We have $N_{h+1}=\sum_{i=1}^{N_h}s^{h}_i$.
Then, $a^h_i$ has $s^{h}_i$ children that are noted 
 $\mC(a^h_{i})=\{a^{h+1}_{i_1,..,i_h,i_{h+1}}: i_{h+1}=1,...,s^{h}_{i}\}$.
As before we can order all the elements of level $h+1$ as 
$\{a^{h+1}_k: k=1,...,N_{h+1}\}$  with  
$k=k(i_1,...,i_h,i_{h+1})=\sum_{j_1=1}^{i_{h}-1}s^h_{j_1}+i_{h+1}$ 
for some $i_{h+1}=1,..,s^{h}_{i}$, being $i=i(i_1,...,i_h)$.  

\medskip

From (\ref{propx1}) we can associate an excursion $x$ to $T^s$
such that $T(x)=T^s$. By construction we have $x\in \X^+(N)$. 
On the other hand for two sequences $s, s'\in \mD(N)$, $s\neq s'$, the
construction made gives different rooted ordered  trees $T^s\neq T^{s'}$ and so 
different associated excursions $x$ and $x'$. 

\medskip

Now, the associated tree $T(x)$ of an excursion $x\in \X^+(N)$, defines 
an element $s\in  \mD(N)$. In fact, the individuals born at every level 
$h\in \{1,...,H(N)-2\}$ of $T(x)$  are totally ordered (by the time of birth), 
let us we enumerate them by $I^h_1,...,I^h_{N_h}$.  The sequence 
of numbers of their children is denoted  by $s^h=(s^{h}_1,...,s^{h}_{N_h})$,
which belongs to $D(N_{h+1},N_h)$ and so $s=(s^1,...,s^{H(N)-2})\in \mD(N)$. 
On the other hand, recall that if  $i<j$ then the rank of every child in 
$\mC(I^h_i)$ is smaller than the rank of every child in $\mC(I^h_j)$, 
at level $h+1$ (see(\ref{presor})), just as in the construction of the tree 
$T^s$.  Hence, we get $T^s=T(x)$.
\end{proof}

\medskip

\section{Transformations on excursions  that are level numbers preserving}

Simple changes on the jumps of an excursion $x\in \X^{+}$ can 
change the level numbers $N(x)$. For instance, 
let $a<\theta(x)$ with $x_a\ge 2$ and
assume $(y_a,y_{a+1})=(1,-1)$. 
Now make a unique change in the sequence of jumps,
instead of $(y_a,y_{a+1})$ put $(y'_a,y'_{a+1})=(-1,1)$. Then, the 
transformed excursion $x'$ satisfies
$N_h(x')=N_{h}(x)$ for $h\notin \{x_{a}-1, x_{a}\}$ and
$N_{a-1}(x')=N_{a-1}(x)+1$, $N_a(x')=N_a(x)-1$. Also
$\PP(X^\Theta=x | X^\Theta\in \X^+ )/\PP(X^\Theta=x' | X^\Theta\in \X^+ )
=(p_a q_{a+1})/(q_a p_{a-1})$.

\medskip

In the next subsections we consider transformations 
$$
\Psi:\X^+\to \X^+, x\to \Psi(x), \; x\in \X^+,
$$
that preserve the level numbers, so   
for all $x\in \X^+$ one has $N(x)=N(\Psi(x))$, that is
$|\I_h(\Psi(x))|=|\I_h(x)|$ for all $h\ge 0$.  
From Proposition \ref{propG0}, these transformations 
are measure-preserving, this means 
$$
\PP(X^\Theta=x | X^\Theta\in \X^+)
=\PP(X^\Theta=\Psi(x) | X^\Theta\in \X^+).
$$
(Similar relations hold for negative excursions). When the set $\I(x)$ is endowed  
with the level order (see (\ref{order})) and since $N(\Psi(x))=N(x)$,
then the set $\I(\Psi(x))$ is also ordered by it and we can write 
$\I(\psi(x))=\{I_{\varphi(i)}: i=0,...,|\I(x)|-1\}$ where 
$\varphi:\{0,...,|\I(x)|-1\}\to \{0,...,|\I(x)|-1\}$ is a bijection   
such that for all $h=1,...,H(x)-1$  the restriction $\varphi_{h}$ 
of $\varphi$ to the indexes of individuals in $\I_h(x)$, 
\begin{equation}
\label{order4}
\varphi_h:\{\sum_{k=0}^{h-1}N_k,\sum_{k=0}^{h}N_k-1\}\to 
\{\sum_{k=0}^{h-1}N_h,\sum_{k=0}^{h}N_k-1\},
\end{equation}
is also a bijection. 

\medskip

In the next subsections we will consider the reversed transformation 
and the shift of bridges, both are level numbers preserving 
transformations. 

\medskip

\subsection{Reversed excursion} 

Let $x\in \X$ be an excursion of length $\theta=\theta(x)$. Then, 
$\R(x)=(x_{\theta(x)-n}: n\in [0,\theta(x)])\in \X$ is a well defined  
excursion called the reversed excursion. From definition 
$\theta(\R(x))=\theta(x)$. The mapping $\R$ defines one-to-one 
mappings in $\X^{+}$ and in $\X^{-}$.

\medskip

We claim that this transformation  $\R$
preserves the associated level numbers, $N(\R(x))=N(x)$. 
It suffices to show it for $x\in \X^{+}$.  Let $x^{*}=\R(x)$.  For
all the notions defined on excursions, we add a superscript ${}^*$ 
when dealing with $x^*$, for instance its jumps and individuals 
are noted by $y^*$ and $I^*$, respectively. 
We have $y^*_n=x^*_{n+1}-x^*_n=
x_{\theta-n-1}-x_{\theta-n}=-y_{\theta-n-1}$.

\medskip

Let us consider an individual $I^*$ of $x^*$ born at time 
$\mb^*\ge 0$ at level $\mh^*\ge 1$ and dying at 
time $\mt^*=\inf\{n>\mb^*: X^*_{n}=h^*\}$. Then, 
$y^*_{\mb^*}=1=-y^*_{\theta-\mb^*-1}$, and so an 
individual $I$ of $x$ dies at time $\mt=\theta-\mb^*$. 
On the other hand $x^*_{\mb^*}=\mh^*$, so $x_{\theta-\mb^*}=\mh^*$,  
then $I$ is born at level $\mh^*$. Finally,
$$
\mt^*=\inf\{n>\mb^*: x^*_n=\mh^*\}=\sup\{\theta-n<\theta-b^*: 
x_{\theta-n}=\mh^*\}. 
$$
Hence, $I$ is born at time 
$\mb=\theta-\mt^*$. Since to each individual $I^*$ of $x^*$ one
is able to associate a unique individual $I$ of $x$ and 
this operation is also invertible, we conclude that the 
associated level numbers are equal, $N(x^*)=N(x)$. 
Note that the lengths of life of the individuals $I^*$ and $I$ 
are the same, in fact they are respectively  $\theta(I^*)=\mt^*-\mb^*$ and 
$\theta(I)=\mt-\mb=\theta-\mb^*-(\theta-\mt^*)$. 

\medskip

In the reversed excursion the time is also reversed, this is why 
the probability of $\R(x)$ is the same as for $x$,  
$\PP(X^\Theta=\R(x) | X^\Theta\in \X^+)
=\PP(X^\Theta=x | X^\Theta\in \X^+ )$.

\medskip
We have that $\R$ is an involution on $\X$, that is
$$
\R\circ \R=\hbox{Id}, \hbox{ the identity on } \X, \hbox{ so } \R^{-1}=\R.
$$
We can also define a random element $\R(X^\Theta)$ by,
$$
\R(X)_{n+1}-\R(X)_n=-Y_{\Theta-n-1}\hbox{  for } n\in [0,\Theta).
$$
By definition $\R(X^\Theta)$ is a random excursion of length $\Theta$ and 
since $\R^{-1}=\R$ one has
$$
\PP(\R(X^\Theta)=x| X^\Theta\in \X^+)
=\PP(X^\Theta=\R(x) | X^\Theta\in \X^+) \hbox{ for } 
x\in \X^+.
$$

Then, $N(\R(X^\Theta)|  X^\Theta\in \X^+ )$ is equally distributed as 
$N(X^\Theta|  X^\Theta\in \X^+)$. 
For the homogeneous random walk with $p\le q$, 
this is a linear fractional random walk of parameter $p$.

\medskip

\subsection{Shift of bridges} 

We will define the class of shift transformations,
$$
\mE=\{\E\{a,b,c;h\}: 1\le a\le b, c\not\in (a,b), h\ge 1\},
$$
where $\E\{a,b,c;h\}$ is partially defined on $\X^+$, its domain  
of definition being
$$
\D(\E\{a,b,c;h\})=\{x\in \X^{+}: 
 0\le a,b,c\le \theta(x), h\le H(x)-1, x_a=x_b=x_c=h \}.
$$ 
Let us describe the action of $\E=\E\{a,b,c;h\}$. Take $x\in \D(\E)$,  
$x[a,b]$ a bridge  with $x_a=x_b=h\ge 1$, and  $c\not\in  (a,b)$ with  
$x_c=h$. In the transformed excursion $\E(x)$, 
the bridge $x[a,b]$ will be shifted 
in $c-b$ units of times when $b\le c$ and in $a-c$ units of
time when $c\le a$.  It is also said that the bridge $x[a,b]$ is inserted 
in $c$. To be precise, for $x\in \X^+$ and $\theta=\theta(x)$, the shift is given by 
\begin{eqnarray}
\nonumber
\hbox{ If } b\le c: &{}& \E(x)[0,a)=x[0,a), \quad \E(x)[c,\theta]=x[c,\theta],\\
\label{defexch}
&{}& \E(x)[a,a+c-b]=x[b,c], \quad \E(x)[c-(b-a),c)=x[a,b) \hbox{ and} \\
\nonumber
\hbox{ If } c\le a: &{}& \E(x)[0,c]=x[0,c], \quad \E(x)(b,\theta]=x(b,\theta],\\
\nonumber
&{}& \E(x)(c,c+(b-a)]=x(a,b], \quad \E(x)[c+(b-a),b]=x[b,c].
\end{eqnarray}
In an equivalent way one can defines $\E(x)$ by
$\E(x)_n=x_{\me_\E(n)}$ with $\me_\E$ a transformation of the 
set of indexes. In the case $b\le c$ it is given by 
\begin{equation}
\label{chti}
\me_\E(n)=
\begin{cases}
n & \hbox{ if } n\in [0,a)\cup [c,\theta],\\
n-(a-b) & \hbox{ if } n\in [a,c-(b-a)]\\
n-(c-b) & \hbox{ if } n\in [c-(b-a),c).
\end{cases}
\end{equation}
Notice that $\me_\E$ is one-to-one and 
$\me_\E(c)\not\in (\me_\E(a),\me_\E(b))$. 
The case $c\le a$ is defined similarly. 

\medskip

From the definition we have $\theta(\E(x))=\theta(x)=\theta$ and 
$N_{h'}(x)=N_{h'}(\E(x))$ for all  $h'\ge 0$. Then
$$
N(\E(x))=N(x).
$$
Now, let $N=(N_h: h\ge 0)$ be (a sequence of) level numbers with $N(0)=1$ 
and finite $\oN$. Let $\theta=2\oN$.  Then,  we have
$$
\E: \X^{+}(N)\cap \D(\E)\to \X^{+}(N).
$$

The (partially defined) shift $\E$ is one-to-one. 
The shifts of $x[a,b]$ to $a$ or to $b$, 
are the identity.  When $b\le c$, $\E^{-1}$ is 
the shift of the bridge $x[c-(b-a),c]$ to $a$, and its domain is 
$\D(\E\{c-(b-a),c,a;h\})$. The equality $\E^{-1}(\E(x))=x$ holds because 
$\E(\D(\E\{a,b,c;h\}))\subseteq \D(\E\{c-(b-a),c,a;h\})$.
If $c\le a$ the inverse $\E^{-1}$ is 
the shift of the bridge $x[c,c+(b-a)]$ to $b$. 

\medskip

Let $\E=\E\{a,b,c;h\}$, $\E'=\E\{a',b',c';h'\}$ be two shifts. Then
the composition of shifts $\E'\circ \E$ is well-defined in the domain 
$\D(\E'\circ \E)$ constituted by the excursions $x\in \X^{+}$  such that 
$0\le a,b,c,a',b',c'\le \theta(x)$, $h,h'\le H(x)-1$ and 
$$
 x_a=x_b=x_c=h, \;\;
x_{\me_\E^{-1}(a')}=x_{\me_\E^{-1}(b')}=x_{\me_\E^{-1}(c')}=h'.
$$
Note that $(\E'\circ \E)^{-1}=\E^{-1}\circ {\E'}^{-1}$.
Therefore, when we consider the composition of transformations 
belonging to $\mE=\{\E\{a,b,c;h\}: 1\le a\le b, c\not\in (a,b), h\ge 1\}$, 
one gets a group of transformations where the domain of definition 
depends on the transformation.

\medskip

\subsection{Shift of bridges and level numbers preserving}
 
Let $x\in \X^{+}$ and $x[a,b]=x(I)$ be an excursion of an individual $I$  
at level $h\ge 1$, so with $x_a=h=x_b$, $x(a,b)>h$. 
The insertion of the excursion $x[a,b]$ in $c$ is simply called the 
insertion of $I$ in $c$. We denote by $\mE^e\subset \mE$ 
the class of shifts of positive excursions. This also acts as a group
of transformation because the inverse of a shift of an excursion is also a 
shift of an excursion. 

\medskip 

We take $\I(x)=\{I_i(x): i=0,..,|\I(x)|-1\}$ ordered by the level order, that is
first by the level and for equal level then 
by the time of birth, see (\ref{order}).
Let $x\in \X^+$ and $x(I_{i_0}(x))=[\mb_{i_0}(x),\mt_{i_0}(x)]$ be 
the excursion of the individual $I_{i_0}(x)$ born at level $\mh_{i_0}(x)$. 
Let $c$ be such that
$x_c=\mh_{i_0}(x)$ and $c\not\in (\mb_{i_0}(x),\mt_{i_0}(x))$. 
Once making the transformation $\E(x)$ that 
inserts $I_{i_0}(x)$ in $c$, the set $\I(\E(x))$ is written   
$\I(\E(x))=\{I_{\varphi(i)}(\E(x)): i=0,...,|\I(x)|-1\}$, where 
$\varphi:\{0,...,|\I(x)|-1\}\to \{0,...,|\I(x)|-1\}$ is a one-to-one mapping  
that preserves the level, $\mh(I_{\varphi(i)}(\E(x)))=\mh(I_i(x))$,
and it induces a one-to-one mapping 
$\I_{h}(x)\to \I_{h}(\E(x))$, see (\ref{order4}). 
The times of birth and death of the individuals in $\E(x)$ is computed 
with (\ref{defexch}) and (\ref{chti}), 
$$
\forall i: \; \mb_{\varphi(i)}(\E(x))=
\me_\E(\mb_i(x)),\; \mt_{\varphi(i)}(\E(x))=\me_\E(\mt_i(x)).
$$
Let us describe the case $c\ge \mt_{i_0}(x)$. The unique change in 
terms of the tree is that the parent (and so the ancestors) of $I_{i_0}(x)$
can change. One has
$$
\mP(I_{\varphi(i_0)}(\E(x)))
\begin{cases}
&=\mP(I_{i_0}(x)) \hbox{ if } x[\mb_{i_0}(x),c]\ge \mh_{i_0}(x) , \hbox{ and }\\
&\neq \mP(I_{i_0}(x)) \hbox{ if } \min[\mb_{i_0}(x),c]\le \mh_{i_0}(x)-1.
\end{cases}
$$  
In the last case one has,
\begin{eqnarray*}
&{}& \mP(I_{\varphi(i_0)}(\E(x))=I_k(x) \hbox{ with } 
\mh(I_k(x))=\mh_{i_0}(x)-1 \hbox{ and }\\ 
&{}& \mb(I_k(x))=\max\{\mb(I_j(x)): \mh(I_j(x))=\mh_{i_0}(x)-1,   
\mb(I_j(x))<\mb(I_{\varphi(i_0)}(\E(x))\}.
\end{eqnarray*}
In terms of trees:  if  $\mP(I_{i_0}(x))=\mP(I_{\varphi(i_0)}(\E(x)))$ then 
$T(x)$ and $T(\E(x))$ coincide, except by the numbering of $I_{i_0}$ in 
the set of children $\mC(\mP(I_{i_0}(x)))$. 
When $\mP(I_{i_0}(x))\neq \mP(I_{\varphi(i_0)}(\E(x)))$, 
the subtree $T(I_0(x))$, that 
hanged from $\mP(I_0(x))$ in $T(x)$, it hangs from  
$\mP(I_{\varphi(i_0)}(\E(x)))$ in $T(\E(x))$, that is once making the shift $\E$.

\medskip

\begin{remark}
\label{rete}
One can exchange the excursions of two individuals 
$I_i$ and $I_j$ at the same level in a way that keeps their ranks $i$ and $j$ 
respectively. This is done
by making a sequence of shifts of their children $\mC(I_i)$ and
$\mC(I_j)$ and so their rank $i$ and $j$ are kept, and after making  
these shifts their new children are $\mC(I_j)$ and $\mC(I_i)$ respectively. 
Hence, their excursions have been exchanged but their ranks are the same.
This property will be used in the proof of the next result.
\end{remark}

\medskip

In this result we will start  from some excursion $x$ and make
a sequence of shifts of excursions, so one has 
$x^r=\E_r\circ...\circ \E_1(x^0)$, $r\ge 1$. 
When writing this composition we always assume that $x$ is 
in the domain of definition of $\E_r\circ...\circ \E_1$. On the other hand 
we recall 
that for any sequence of level numbers $N=(N_h: h\ge 0)$ with $\oN$
finite, there exists some excursion $x$ with $N(x)=N$.

  \medskip

\begin{proposition}
\label{mainprop}
Let $x\in \X⁺$ be a positive excursion with level numbers $N(x)$. Then, 
$$
\{x' \in \X^+: N(x')=N(x)\}
=\{x^r=\E_r\circ...\circ \E_1(x): \E_1,..,\E_r\in \mE^{e},  r\ge 1\}.
$$ 
\end{proposition}

\begin{proof}
The statement is equivalent to the following one:  if  $x,x'\in \X^{+}$ 
are two positive excursions  with the same level numbers $N(x)=N(x')$,
 then there is a sequence of shifts of excursions $(\E_i: i=1,...,r)$ such that 
\begin{equation}
\label{seqch}
x'=\E_r\circ...\circ\E_1(x).
\end{equation}
Let us prove it. We have $N_h(x)=N_h(x')$ for all $h\ge 0$, so the heights of $x$ 
and $x'$ are equal, $H(x)=H(x')=H$. Then $N_{H+k}(x)=0=N_{H+k}(x')$ 
for all $k\ge 0$. Put $N_h=N_h(x)$ for $h=0,..,H-1$.  We will 
show (\ref{seqch}) by an induction procedure on pairs of consecutive levels. 

\medskip

Let us consider levels $1$ and $2$. There are $N_1$ individuals at level $1$ in 
both $x$ and $x'$. Let $I_1,..,I_{n_1}$ and $I'_1,..,I'_{N_1}$ be the individuals 
in $x$ and $x'$ respectively born at level $1$.  For $i=1,...,N_1$, let $s_{i}$ 
and $s'_{i}$ be respectively the number of 
children of $I_i$ and $I'_i$ in $x$ and $x'$, respectively.  We have, 
\begin{equation}
\label{eq1x}
\sum_{i=1}^{N_1}s_{i}=N_2=\sum_{i=1}^{N_1}s'_{i}.
\end{equation}
Let $k_1$ be the maximal integer $1\le k_1\le N_1$ for which there exists 
a permutation $\pi$ of $\{1,...,N_1\}$ that satisfies $s_{\pi(i)}=s'_i$ 
for all $i=1,..,k_1$.  As indicated in  Remark \ref{rete} we shift the excursions 
of the children of the individuals $I_i$ and $I_{\pi(i)}$ in order that after
these shifts,  the children of $I_i$ and $I_{\pi(i)}$ 
are $\mC(I_{\pi(i)})$ and $\mC(I_i)$, respectively. Hence, we have $s_i=s'_i$
for all $i=1,...,k_1$. 

\medskip

If $k_1=N_1$ we go to the next step. Assume $k_1<N_1$. If 
$s_{k_1+1}>s'_{k_1+1}$, then we insert $s_{k_1+1}-s'_{k_1+1}$ 
excursions of children of $I_{k_1}$ into the ending 
coordinate of the excursions of the children 
of some individual $I_r$ with $r>k_1+1$.  
If $s_{k_1+1}<s'_{k_1+1}$ then we take $s'_{k_1+1}-s_{k_1+1}$ 
excursions of the children of the individuals $(I_r: r>k_1+1)$ and insert them 
at the coordinate marking the end of the excursions of the children of 
$I_{k_1+1}$. These changes 
can be done because (\ref{eq1x}) holds. So, we get $s'_i=s_i$, $i=1,..,k_1+1$. 
Our algorithm continues  and so by simply shifting excursions of lervel $2$
in $x$, one gets $s_i=s'_i$ for $i=1,...,N_1$.
Then we go to the next step of the induction. 

\medskip
 
Suppose for some $h\ge 2$ and for all levels $1\le h'< h$ one has 
$s_{i,h',k}=s'_{i,h',k}$ for all $k=1,...,N_{h'}$, 
where  $s_{i,h',k}$ and $s'_{i,h',k}$ are the number of children of individuals 
$I_i$ and $I'_i$ in $x$ and $x'$ respectively, at level $h'<h$. 
Let us show that we can make a sequence of shifts at level 
$h+1$ in $x$, in order that this is also satisfied up to level $h+1$. 
Let $I_r$ and $I'_r$, $r=1,...,N_h$ be the individuals of 
$x$ and $x'$ at level $h$ and let $s_r$ and $s'_r$ be the number 
of their children. Since, similarly to (\ref{eq1x}), we have
$\sum_{r=1}^{N_h}s_r=N_{h+1}=\sum_{r=1}^{N_h}s'_r$, 
the same argument as the one made for individuals of levels $1$ and $2$  
works in this case for individuals of levels $h$ and $h+1$,
and shows that we can permute the children of the individuals 
of level $h$ in $x$, to get $s_r=s'_r$ for $i=1,...,N_h$. Hence, the
induction step holds.

\medskip

This induction states that there is sequence of shifts of excursions 
$(\E_i:  i=1,...,d)$ that once applied to $x$ gives:
$$
\hx=F(x) \hbox{ with }F=\E_d\circ...\circ\E_1  \hbox{ and  } 
\hx \hbox{ satisfies }T(\hx)\equiv T(x'),
$$
that is, the excursions  $\hx$ and $x'$  have equivalent trees. 
Then,  
|
by using property (\ref{propx1}), the class of children $(\mC(I): I\in \I(\hx))$ 
of $T(\hx)$ can be ordered inductively as in $(\mC(I): I\in \I(x'))$. 
Then, there exists a 
sequence of shifts of excursions $(\E_i:  i=d+1,...,d+d')$ fulfilling
$$
x'=F'(\hx) \hbox{ with } F'=\E_{d+d'}\circ...\circ \E_{d+1},
$$ 
because the change of enumeration corresponds in making  
shifts of excursions 
(that do not change the predecessors). We have proven
\begin{equation}
\label{esc1}
x'=F'\circ F(x)=\E_{d+d'}\circ...\circ \E_{d+1}\circ \E_d\circ...\circ\E_1(x). 
\end{equation}
 \end{proof}

\medskip

From (\ref{esc1}) we also find,
$x=F^{-1}\circ {F'}^{-1}(x')=\E_1^{-1}\circ...\circ \E_{d}^{-1}
\circ \E_{d+1}^{-1}...\circ \E_{d+d'}^{-1}(x')$,
where the ${\E_i}^{-1}$, $i=1,...,d+d'$ are also shifts of excursions.

\medskip

\begin{remark}
\label{contar}
Proposition \ref{mainprop} states that the action of 
shifts of excursions allows to retrieve the whole class 
of trees $\{T\}(N(x))$ that preserves $N(x)$ 
and when we know that two excursions $x$ and $x'$ have the same 
level numbers $N(x)=N(x')$ the proof provides an algorithm of 
transforming one excursions into the other one. 
\end{remark}

\medskip

{\bf Example}. Let $x\in \X⁺$, $x'=\E^*(x)$ where $\E^*$ is
the  shift of a negative excursion $x[a,b]$ to $c$.
So, in this case $c\not\in (a,b)$, $x_a=x_b=x_c=h\ge 1$ and 
$0<x(a,b)< h$. Then,  Proposition  \ref{mainprop} states that
one can write $x'=\E_r\circ...\circ \E_1(x)$ for a sequence 
of shifts of positive excursions $(\E_1,...,\E_r)$.

\medskip

\section{The Vervaat transform} 

We wish that the Vervaat transform of an excursion is also an excursion,
so we define it only on the set of excursions attaining their heights 
at a single coordinate. As for the other transformation we seek  
to see how it changes the level numbers. 

\medskip

The Vervaat transform, introduced in \cite{Ve}, has been mainly studied 
for the Brownian excursions and bridges. In \cite{lpt} it is studied 
for random walks, and
constructed  from the coordinate where the walk attains its global minimum.

\medskip

\subsection{Definition and properties}

Recall $H(x)$ denotes the height of an excursion $x$.  
Let us consider
$$
\X^{+,U}=\{x\in \X^+: N_{H(x)-1}=1\}, \; 
\X^{-,U}=\{x\in \X: N_{H(x)+1}=1\},\; X^U=\X^{+,U}\cup \X^{-,U}.
$$
$\X^{+,U}$ (respectively $\X^{-,U}$) is the set of positive (respectively negative) 
excursions for which the height is attained at a 
unique coordinate. In fact, for $x\in \X$  there 
exists a unique $m\in (0,\theta)$ such that $x_m=H(x)$ if and only if
$N_{H(x)-1}(x)=1$. The same for negative excursions.

\medskip

We will introduce the Vervaat transform  $\V:\X^U\to X^U$
and give some of its properties in detail.  
Let $X\in \X^{+,U}$ with length $\theta=\theta(x)$. 
It is useful to denote the transformed point by $\hx=\V^+(x)$ and to add to
all the notions associated to $\hx$ a hat $\;\, \widehat{}$, 
for instance its jumps are written 
$\hy_n$ and the individuals $\hI$. Let us make the definition,
\begin{equation}
\label{vv0}
\hx_n=
\begin{cases}
x_{n+m}-H(x) & \hbox{ if } 0\le n \le \theta-m\\
x_{n+m-\theta}-H(x) & \hbox{ if } \theta-m+1\le n \le \theta.
\end{cases}
 \end{equation}
Then, $\hx_0=0=\hx_{\theta}$. 

\medskip

\begin{proposition}
\label{vv1}
The Vervaat transform  $\V:\X^{U}\to \X^{U}$ is a bijection 
satisfying $\theta(\V(x))=\theta(x)$,  $H(\V(x))=-H(x)$ and
\begin{equation}
\label{vv5}
\V:\X^{+,U}\to \X^{-,U} \hbox{ and } \V:\X^{-,U}\to \X^{+,U}.
 \end{equation}
Moreover,
\begin{equation}
\label{vv3}
N_{-h}(\V(x))=
\begin{cases}
N_{H(x)-1-h}(x) & \hbox{ for } h=0,..,H(x)-1 \hbox{ when } x\in \X^+, \\
N_{H(x)+1-h}(x) & \hbox{ for } h=H(x)+1,...,0 \hbox{ when } x\in \X^-.
\end{cases}
\end{equation}
Finally, $\V$ is an involution in $\X^{U}$, that is
\begin{equation}
\label{vv4}
\V\circ \V ={\hbox{Id}}_{\X^{U}} \hbox{  the identity on } X^{U} \hbox{ and }  \V^-={\V}^{-1}.
\end{equation}
\end{proposition}

\begin{proof}
From definition $\theta(\hx)=\theta$ and $H(\hx)=-H(x)$. 
Moreover,  $\hx_{\theta-m}=x_{\theta}-H(x)=-H(x)$ and one can check 
that $\hx$ attains its height $H(\hx)=-H(x)$ at the unique coordinate 
$\hm=\theta-m$, so $\hx\in \X^U$.  
From the definition $\hx(0,\theta)<0$ if $x\in \X^+$ and
$\hx(0,\theta)<0$ if $x\in \X^+$, so (\ref{vv5}) holds.

\medskip

Let us turn to the proof of (\ref{vv3}).
The jumps $\hy_n=\hx_{n+1}-\hx_n, \, 0\le n<\theta$,  satisfy  
$$
\hy_n=
\begin{cases}
y_{n+m} & \hbox{ if } 0\le n \le  \theta-m\\
y_{n+m-\theta} & \hbox{ if }  \theta-m+1\le  n < \theta.
\end{cases}
$$
Note that $\hy_{\theta-m}=\hx_{\theta-m+1}-\hx_{\theta-m}
=x_1-x_\theta=y_0=1$.

\medskip

Now take $x\in \X^{+,U}$ but the same argument can be done
for $x\in \X^{-,U}$).Let $h\ge 1$. An individual $\hI$ is born at time 
$\hmb\ge 0$ if at level $-h$ if $\hy_{\hmb}=-1$ and at level
$-h$ if $\hx_{\hmb}=-h$  
(notice that $\hmb\neq \theta-m$ because $\hy_{\theta-m}=1$).  So, 
$\hI$ is born at time $\hmb$ in $x$ if an individual $I$ in $x$ dies 
at time 
$$
\begin{cases} 
\hmb+m+1 & \hbox{ if } 0\le \hmb< \theta-m \\ 
\hmb+m-\theta+1 & \hbox{ if }  \theta-m+1 \le \hmb < \theta,
\end{cases}
$$
and it is born at level $H(x)+\hx_{\hmb+1}=H-h-1$ because $\hx_{\hmb+1}
=\hx_{\hmb}+\hy_{\hmb}=-h-1$.
Therefore, the level numbers $\hN(\hx)$ has $H(x)$ 
levels $(N_{-h}(\hx): 0\le h\le H(x)-1)$ as $N(x)$ and it satisfies 
(\ref{vv3}). In particular $\hN_0(\hx)=1=N_{H(x)-1}(x)$ (the last equality 
being the hypothesis). 

\medskip

Let us now show (\ref{vv4}). We also make the argument only for $x\in \X^{+,U}$. 
Let us put $\wx=\V(\V(x))$. Its height $H(\hx)=-H(x)$ is attained at the unique 
coordinate $\hm=\theta-m$. From (\ref{vv0}) we find, 
$$
\wx_n=
\begin{cases}
\hx_{n+m}+H(x) & \hbox{ if } 0\le n \le \theta-\hm\\
\hx_{n+\hm-\theta}+H(x) & \hbox{ if } \theta-\hm+1\le n \le \theta.
\end{cases}
$$
Then, by using $\theta-\hm=m$, 
$n+\hm=n+\theta-m$ and $n+\hm-\theta=n-m$, we get  
$$
\wx_n=
\begin{cases}
\hx_{n+\theta-m}+H(x) & \hbox{ if } 0\le n \le m\\
\hx_{n-m}+H(x) & \hbox{ if }  m+1\le n\le \theta.
\end{cases}
$$
By using again the definition (\ref{vv0}) we get that 
$\wx_{l+m-\theta}=x_{l-m-\theta}$ for $l=\theta-m+1,...,\theta$ 
and $\wx_{l+m}=x_{l+m}$ for $l=1,..,\theta-m$. This gives
$\wx_n=x_n$ for $n=1,..,m$ and $l=m+1,...,\theta$. 
Now, for $n=0$ we have
$\wx_0=\hx_{\theta-m}=-H(x)+H(x)=0=x_0$. 
Then $\wx=x$, that is $\V(\V(x))=x$. This implies 
that $\V$ is a bijection and that $\V^{-1}=\V$.
\end{proof}

\medskip

In the previous construction of $\V$ we have collapsed times $0$ and 
$\theta$,  and we have inserted an interval of times of length
$\theta$ just after the coordinate $m$ where the height is attained. 
So, $m$ plays the role of $0$, there are
$\theta$ units of times starting from $m$ because we identify 
$0$ and $\theta$. Let us see that at $m$ is the unique place where 
we can do this procedure. 
Take $x\in \X$ with $\theta(x)=\theta$,  fix $k\in [0,\theta]$ and 
define $x'$ by,
\begin{equation}
\label{transl}
x'_n=
\begin{cases}
x_{n+k}-x_k & \hbox{ if } 0\le n \le \theta-k\\
x_{n+k-\theta}-x_k & \hbox{ if } \theta-k+1\le n \le \theta.
\end{cases}
\end{equation}
Then $x'_{0}=0=x'_\theta$ holds. We develop the argument only for
$x\in \X^+$. We have that $x'$
is a positive excursion only when $k=0$ or $k=\theta$ and in both cases 
$x'=x$. Moreover, $x'$ is a negative excursion 
only if there exists a unique $m$ for which $x_m=H(x)$ and $k=m$.
In fact from (\ref{transl}) we must have $x_n\le x_k$ for all $n\in 
[0,\theta]$, so $x_k=H(x)$. If for some other $k'\in (0,\theta)$ one has 
$x_{k'}=H(x)$, then it would not be a negative excursion because in 
(\ref{transl}) one would get three different times $n$ at which one should 
have $x'_n=0$. Now, if there is a unique point, the unique negative 
excursion is defined with $k=m$ as in the previous construction. 

\medskip

We notice that for $x\in X^{U}$, 
$N(\V(x))$ only depends on $N(x)$. Moreover, if $N(x)=N(x')$
for  $x,x'\in \X$, then $\V$ is defined for none or for both $x,x'$, 
and when it is defined then $N(\V(x))=N(\V(x'))$. 

\medskip

The Vervaat transformed $\V(X^\Theta)$ excursion random variable,
when $X^\Theta\in \X^U$, is made with the set of jumps
$\hY_n=\V(X^\Theta)_{n+1}-\V(X^\Theta)_n, \, 0\le n<\Theta$ given by,   
$$
\hY_n=
\begin{cases}
Y_{n+m} & \hbox{ if } 0\le n \le  \Theta-m\\
Y_{n+m-\theta} & \hbox{ if }  \Theta-m+1\le  n < \Theta,
\end{cases}
$$
where $m$ is the coordinate at which $X^\Theta_m=H(X^\Theta)$.

\medskip

In the homogeneous case, with $p\le q$, 
$N(X^\Theta | X^\Theta\in \X^{+,U})$ is distributed as a 
linear fractional Galton Watson process of parameter $p$ 
conditioned to $N_{H(X^\Theta)=1}$.  We shall describe 
$N(\V(X^\Theta) | X^\Theta\in \X^{+,U})$ at the end of the last Section.

\medskip

Let us summarize the global action of some of the transformations already 
introduced.  We recall the sign reflection $-$ on $\X$, given by $x\to -x$. 
We have that $-$ and $\R$ preserve $\X^U$, that is $-\X^U=\R(\X^U)=\X^U$. 
Moreover $-$, $\R$ and $\V$ are involutions on $\X^U$, that is 
$$
{-}^2=\R^2=\V^2=\hbox{Id}_{\X^U},
$$ 
and they commute  among themselves:
$$
-\R=\R\circ {-}, {-}\V=\V\circ {-},
\R\circ\V=\V\circ \R.
$$
So, the group of transformations $\G(-,\R,\V)$ generated by these 
transformations acting on $\X^U$ 
(we denote by $\R$ the restriction of $\R$ to $\X^U$),  has eight elements 
$$
G(-,\R,\V)=\{\hbox{Id}_{\X^U},-\hbox{Id}_{\X^U},\R, -\R, \V,-\V, \V\circ \R,-\V\circ\R\}.
$$

\medskip

\subsection{A martingale}

We will compute $\PP(X^\Theta\in \X^U$, that is the probability 
of the set of excursions that attain its 
height at a unique point, which is the domain of the Vervaat
transformation. To this purpose it is useful to introduce a martingale.

\medskip

The random walk $X=(X_n: n\ge 0)$ starts  from $X_0=c$, for some 
fixed $c\in \ZZ$, but as said when the computations 
refers to the excursion random 
variable $X^\Theta$ or to $\Theta$, we always take $X_0=0$,  
or $X_0=1$ when dealing with positive excursions. 
To  define a martingale $Z=(Z_n: n\ge 0)$ associated to 
$X$ starting from $c$ we first introduce the sequence 
$(A_n: n\in \ZZ)$ given by $A_0=0$ and for $i\ge 1$,
$$
A_i=\sum_{j=1}^{i} \prod_{k=0}^{j-1}(q_{c+k}/p_{c+k}),\;\;
A_{-i}=-\! \!\sum_{j=-(i-1)}^{0} 
\prod_{k=j}^{-1} (p_{c+k}/q_{c+k})\;.
$$
Note that $A_{-1}=-1$.  The sequence $(A_i: i\in \ZZ)$ is similar to 
$(\alpha_i: i\in \ZZ_+)$ introduced in (\ref{eqG1}), but shifted in 
the initial coordinate, and it corresponds to the set of coordinates
used to immerse the random walk in the Brownian motion as done 
in \cite{Harris}. 
The process $Z$, which is proven to be a martingale, is defined by 
$$
Z_{n}=A_{X_n-c}, \; n\ge 0, \hbox{ and so } Z_0=0.
$$
Since the sequence $(A_n: n\in \ZZ)$ is strictly increasing,
$j\to A_{j-c}$ for $j\in \ZZ$, is one-to-one. On the other hand 
when $X_n=k$ then $Z_n=A_{k-c}$. 
So, the information $\{Z_i: 0\le i\le n\}$ 
is equivalent to $\{X_i: 0\le i\le n\}$ and so it also gives
$\{p_{X_i}, q_{X_i}: 0\le i\le n\}$. We have
$$
\EE(Z_{n+1} | Z_{i}, i\le n)=p_{X_n}A_{X_n-c+1}+q_{X_n}A_{X_n-c-1}.
$$
It is straightforward to see that $(Z_n: n\ge 0)$ is a martingale.
Let us check it when $Z_n=A_{j}$ for some $j\ge 1$. We have 
$X_n=c+j$, so
\begin{eqnarray*}        
&{}&\EE(Z_{n+1} | Z_{i}, i\le n)=p_{c+j} A_{j+1}+q_{c+j} A_{j-1}\\
&{}& 
=p_{c+j}\left(A_j+\left(\prod_{k=0}^{j-1}\frac{q_{c+k}}{p_{c+k}}\right)
\frac{q_{c+j}}{p_{c+j}}\right)
+q_{c+j}\left(A_j-\left(\prod_{k=0}^{j-1}
\frac{q_{c+k}}{p_{c+k}}\right)\right)=A_j=Z_n.
\end{eqnarray*}


Define the stopping time 
$\tau_l=\inf\{n\ge 0: X_n=l\}$. For $r<c<s$, consider the stopping time
$\tau=\inf(\tau_r,\tau_s)$. By using the martingale property of 
$Z_{n}=A_{X_n-c}$ we get,
$$
\PP(\tau_s<\tau_r)A_{s-c}+(1-\PP(\tau_s<\tau_r)A_{r-c}=0.
$$
Then,
\begin{equation}
\label{eqbas}
\PP(\tau_s<\tau_r)=\frac{-A_{r-c}}{A_{s-c}-A_{r-c}} \hbox{ and }
\PP(\tau_r<\tau_s)=\frac{A_{s-c}}{A_{s-c}-A_{r-c}}\,.
\end{equation}
To compute this quantities let us consider $u,v\ge 1$. From 
the definition of $A_u, A_v$ and after multiplying and dividing by
$\prod_{k=-(u-1)}^{-1} \frac{q_{c+k}}{p_{c+k}}$ one gets
\begin{eqnarray*}
&{}&\frac{-A_{-u}}{A_v-A_{-u}}=\frac{\sum_{j=-(u-1)}^{0} 
\prod_{k=-(u-1)}^{j-1}\frac{q_{c+k}}{p_{c+k}}}
{\sum_{j=-(u-1)}^{v} \prod_{k=-(u-1)}^{j-1}\frac{q_{c+k}}{p_{c+k}}}\,,\\
&{}&\frac{A_{v}}{A_v-A_{-u}}=\frac{\sum_{j=1}^{v}
\prod_{k=-(u-1)}^{j-1}\frac{q_{c+k}}{p_{c+k}}}
{\sum_{j=-(u-1)}^{v} \prod_{k=-(u-1)}^{j-1}\frac{q_{c+k}}{p_{c+k}}}\,.
\end{eqnarray*}
Hence, by taking into account that $s-c\ge 1, -(r-c)\ge 1$, we replace the
quantities into (\ref{eqbas}) to get,
\begin{eqnarray}
\nonumber
&{}&\PP(\tau_s<\tau_r)=\frac{-A_{r-c}}{A_{s-c}-A_{r-c}}
=\frac{\sum_{j=r-c+1}^{0}\prod_{k=r-c+1}^{j-1}\frac{q_{c+k}}{p_{c+k}}}
{\sum_{j=r-c+1)}^{s-c} \prod_{k=r-c+1)}^{j-1}\frac{q_{c+k}}{p_{c+k}}}\,,\\
\label{expmart2}
&{}&\PP(\tau_r<\tau_s)
=\frac{A_{s-c}}{A_{s-c}-A_{r-c}}=\frac{\sum_{j=1}^{s-c}
\prod_{k=r-c+1}^{j-1}\frac{q_{c+k}}{p_{c+k}}}   
{\sum_{j=r-c+1}^{s-c} \prod_{k=r-c+1)}^{j-1}\frac{q_{c+k}}{p_{c+k}}}\,.
\end{eqnarray}

\medskip

\subsection{Distribution of the height}
\label{sshd}

Our first result gives the distribution of 
the height $H(X^\Theta)$ of a positive random excursion, 
so with $X_0=0, X_1=1$. (For negative excursions similar computations 
can be made). Firstly, we will get
$$
\PP_1(H(X^\Theta)\ge s, \Theta<\infty) \hbox{ for } s\ge 1.
$$
Recall that (\ref{eqG2}) states that  
$\beta_s=\PP(\Theta < \infty \, | \, X_k=s, \Theta>k)$ (with $k-s\in 2\ZZ_+$), 
satisfies $\beta_s=1$ 
when $\alpha_\infty=1$ and 
$\beta_s=1-\frac{\alpha_{s}}{\alpha_{\infty}}$ if 
$\alpha_\infty<\infty$.
Notice that, 
$\PP_1(\Theta<\infty, H(X^\Theta)\ge 1)=\PP_1(\Theta < \infty)=\beta_1$. 
We will also  compute the probability that at some  
unique coordinate the height of the excursion is attained, this is
$\PP_1(\Theta<\infty, X^\Theta\in \X^{+,U})$.
Notice that
$$
\PP_1(H(X^\Theta)=1, X^\Theta\in \X^{+,U})=q_1, 
$$
and this event corresponds to the excursion $(0,1,0)$. 

\medskip

\begin{proposition}
\label{propxx}
We have,
\begin{equation}
\label{height1}
\PP_1(\Theta<\infty, H(X^\Theta)\ge s)=\frac{\beta_s} {\sum_{j=0}^{s-1} 
\prod_{k=0}^{j-1}\frac{q_{k+1}}{p_{k+1}}},
\end{equation}
and
\begin{eqnarray}
\nonumber
&{}&\PP_1(\Theta<\infty, X^\Theta\in \X^{+,U}, H(X^\Theta)=s)\\
\label{height2}
&{}&= \frac{q_s}
{\sum_{j=0}^{s-1} \prod_{k=0}^{j-1}\frac{q_{k+1}}{p_{k+1}}}  
\times
\frac{\prod_{k=-s+2}^{0}\frac{q_{k}+s-1}{p_{k+s-1}}}
{\sum_{j=-s+2}^{1} \prod_{k=-s+2)}^{j-1}\frac{q_{c+k}}{p_{c+k}}}.
\end{eqnarray}
\end{proposition}

\begin{proof}
One can check that the relations 
(\ref{height1}) and (\ref{height2}) hold for $s=1$. 
Let us prove the relations for $s\ge 2$. We have
$$
\PP_1(\Theta <\infty, H(X^\Theta)\ge s)= 
\PP_1(\tau_s<\tau _0)\PP_s(\Theta <\infty).
$$
Hence by putting $c=1$, $r=0$ one finds
\begin{equation*}
\PP_1(\Theta <\infty, H(X^\Theta)\ge s)
=\frac{\beta_s}
{\sum_{j=0}^{s-1} \prod_{k=0}^{j-1}\frac{q_{k+1}}{p_{k+1}}},
\end{equation*}

On the other hand,
$$
\PP_1(\Theta <\infty, H(X^\Theta)=s, N_{s-1}=1=s)
=\PP_1(\tau_s<\tau_0)\, q_s \, \PP_{s-1}(\tau_0<\tau_s).
$$
By making again use of 
(\ref{expmart2}) one gets
\begin{eqnarray*}
&{}&
\PP_1(\Theta <\infty, H(X^\Theta)=s, N_{s-1}=1)\\
&{}&= \frac{q_s}
{\sum_{j=0}^{s-1} \prod_{k=0}^{j-1}\frac{q_{k+1}}{p_{k+1}}}  \times
\frac{\prod_{k=-s+2}^{0}\frac{q_{k}+s-1}{p_{k+s-1}}}
{\sum_{j=-s+2}^{1} \prod_{k=-s+2}^{j-1}\frac{q_{c+k}}{p_{c+k}}}.
\end{eqnarray*}
Since $\{H(X^\Theta)=s, N_{s-1}=1\}=\{H(X^\Theta)=s,X^\Theta\in \X^{+,U}\}$, 
the result follows.
\end{proof}

\medskip

In the homogeneous case, Proposition \ref{propxx} gives:

\medskip

(i) When $p\neq q$:
\begin{eqnarray*}
&{}& \PP_1(\Theta <\infty, H(X^\Theta)\ge s)
=\beta_s p^{s-1} \left(\frac{q-p}{q^s-p^s}\right)
\hbox{ and }\\
&{}& \PP_1(\Theta<\infty, H(X^\Theta)=s, N_{s-1}=1)
=q(pq)^{s-1}\left(\frac{q-p}{q^s-p^s}\right)^2.
\end{eqnarray*}  

In the case $p\le q$ one has $\beta_s=1$
and if $p>q$ then $\beta_s=(p/q)^s$, see (\ref{eqG3}).

\medskip

(ii) When $p=q$ we have $\PP(\Theta<\infty | X_0=1)=1$ and we get: 
\begin{eqnarray*}
&{}& \PP_1(\Theta<\infty, H(X^\Theta)\ge s)=\frac{1}{s},\\
\nonumber
&{}& \PP_1(\Theta<\infty, H(X^\Theta)=s, X^\Theta\in \X^{+,U})
=\frac{1}{2s^2} \hbox{ and }\\
&{}& \PP_1(\Theta<\infty, X^\Theta\in \X^{+,U})
=\frac{\pi^2}{12}.
\end{eqnarray*}

\medskip

For negative excursions we set $X_1=-1$ and the same computations 
give the analogous results for the $H(X^\Theta)$. 
In the case $p\neq q$ one 
must exchange $p$ by $q$ in the formulae  
and note that when $q\le p$ one has $\PP_{-1}(\Theta<\infty)=1$. 

\medskip

\section{Doob transform and conditioning on always return to $0$}  
\label{ssdoob}

Now we study the probability behavior of trajectories that return to $0$, 
so for which $X^\Theta$ is well defined. 
All we will do in the next paragraphs has a meaning when 
$\PP_0(\Theta<\infty)<1$. Since
$\PP_0(\Theta<\infty)=
p_0 \PP_1(\Theta<\infty)+q_0\PP_{-1}(\Theta<\infty)$,
we will be in the case 
$\PP_1(\Theta<\infty)<1$ or  $\PP_{-1}(\Theta<\infty)<1$.
When this happens we will define the jump probabilities ensuring 
to have a.s. return to the origin, and this will give as a byproduct the
statistics of $X^\Theta$ conditioned to $\Theta<\infty$. 

\medskip

The first return to $0$ is called $\tau^+=\inf\{n>0: X_n=0\}$ (this
variable was denoted $\Theta$ when starting from $X_0=0$). 
We also consider the hitting time 
of $0$, $\tau(\om)=\inf\{n\ge 0: X_n=0\}$. So,
if $X_ 0\neq 0$ one has $\tau^+=\tau$. 

\medskip

Now, we define a canonical random walk 
that always return to $0$. First, from the Markov property we have for all 
$i\in \ZZ$,
\begin{equation}
\label{QP1}
\PP_i(\tau^+<\infty)
=p_i \PP_{i+1}(\tau<\infty)+ q_i \PP_{i-1}(\tau<\infty).
\end{equation}
In particular and since $\PP_0(\tau_0<\infty)=1$,
$$
\PP_1(\tau^+<\infty)
=q_1+p_1 \PP_2(\tau<\infty),\quad
\PP_{-1}(\tau^{+}<\infty)=p_{-1}+q_{-1} \PP_{-2}(\tau<\infty).
$$
Now, let $i_0=0$ and $i_1,...,i_n$ be all positive or all negative, 
and such that $|i_{k+1}-i_k|=1$ for $k=0,...,n-1$. From the Markov 
property one has,  
\begin{eqnarray*}
\PP_0(X_k=i_k, k=1,...,n, \tau^+<\infty)&=&
\PP_{i_n}(\tau^+<\infty)\PP_{0}(X_k=i_k, k=1,...,n)\\
&=&\PP_{i_n}(\tau^+<\infty)\prod_{k=0}^{n-1}p(i_k,i_{k+1}).
\end{eqnarray*}
So,
\begin{eqnarray}
\nonumber
\PP_0(X_k=i_k, k=1,...,n, \tau^+<\infty | \tau^+<\infty)
&=&\frac{\PP_{i_n}(\tau<\infty)}{\PP_0(\tau⁺<\infty)} 
\left(\prod_{k=0}^{n-1}p(i_k,i_{k+1})\right)\\
\label{expprod}
&=&\prod_{k=0}^{n-1}p(i_k,i_{k+1})\frac{\PP_{i_{k+1}}(\tau<\infty)}
{\PP_{i_k}(\tau^+<\infty)}.
\end{eqnarray}
Let us define,
$$
\tp_i=p_i \frac{\PP_{i+1}(\tau<\infty)}{\PP_i(\tau^+<\infty)}, \quad
\tq_i=q_i \frac{\PP_{i-1}(\tau<\infty)}{\PP_i(\tau^+<\infty)}.
$$
From (\ref{QP1}) we get $\tp_i+\tq_i=1, \; i\in \ZZ$.

\medskip

Let us endow the random walk $X=(X_n : n\ge 0)$ 
with the transition probabilities $\tp_i\in (0,1)$ and $\tq_i=1-\tp_i$
for the passages from $i$ to $i+1$ and from $i$ to $i-1$ respectively,
for $i\in \ZZ$. The transition probability matrix given by $\tp_i$, $\tq_i$, 
is called the $Q-$matrix 
(of always return to $0$). Let $\tPP_k$ be the probability law 
when using $\tp_i$ and $\tq_i$ for $i\in \ZZ$ 
and when the walk starts from $X_0=k$. Let us write  $\tp_i, \tq_i$, 
in terms of known quantities.

\medskip

The hitting probabilities for excursions are the
$(\beta_i : i\in \ZZ\setminus \{0\})$ defined 
in (\ref{eqG2}) and (\ref{eqG2x}).  We have 
$\PP_i(\tau_0^{+}<\infty)=\PP_i(\tau_0<\infty)=\beta_i$ for $i\in
\ZZ\setminus \{0\}$  and we define
\begin{equation}
\label{beta0}
\beta_0=\PP_0(\tau_0^{+}<\infty)=p_0 \beta_1+q_0 \beta_{-1}.
\end{equation}
Hence the transition probabilities of the $Q-$matrix are given by
\begin{eqnarray}
\label{wprobtr}
&{}& \tp_i=p_i \frac{\beta_{i+1}}{\beta_i},\; \tq_i=q_i 
\frac{\beta_{i-1}}{\beta_i} \hbox{ for } |i|\neq 1,\\
\nonumber
&{}& \tp_1=p_1 \frac{\beta_{2}}{\beta_1},\; \tq_1=q_1
\frac{1}{\beta_1} \hbox{ and }
\tp_{-1}=p_{-1} \frac{1}{\beta_{-1}},\; \tq_{-1}=q_{-1}
\frac{\beta_{-2}}{\beta_{-1}}.
\end{eqnarray}

Let $x=(i_0,...,i_n)\in \X$ be an excursion, so with $i_0=0=i_n$,  
$i_1,...,i_n$ all positive or all negative,
and such that $|i_{k+1}-i_k|=1$ for $k=0,...,n-1$. 
From (\ref{expprod}) and since $\PP_{i_n}(\tau<\infty)=
\PP_0(\tau<\infty)=1$, we find,
\begin{eqnarray}
\nonumber
\tPP_0(X_1=i_1,...,X_n=i_n)&=&\left(\prod_{i=0}^{n-1} p(i_{k-1},i_k)\right)
{\PP_0(\tau^+<\infty)}^{-1}\\
\label{doobx}
&=& \PP_0(X_1=i_1,...,X_n=i_n){\PP_0(\tau^+<\infty)}^{-1}.
\end{eqnarray}

By summing over all the excursions  $x=(i_0,...,i_n)\in \X$ in (\ref{doobx}), we get   
$$
\tPP_0(\Theta<\infty)=\PP_0(\tau^+<\infty){\PP_0(\tau^+<\infty)}^{-1}=1.
$$
Let us note by $\tEE$ the mean expected value associated 
to $\tPP$. From (\ref{doobx}) we also get that for every nonnegative function 
$g:\NN\to \RR_+$ it holds
$$
\tEE_0(g(\Theta))
=\sum_{n\ge 1}g(n) \PP_0(\tau^+=n){\PP_0(\tau^+<\infty)}^{-1}
=\EE_0(g(\Theta) \, | \,  \Theta<\infty).
$$

\medskip

\begin{proposition}
\label{propdoob}
We have  
\begin{equation}
\label{doobz1}
\PP_0(X^\Theta)=x \, | \, \Theta<\infty)=\tPP_0(X^\Theta)=x) ,\,  x\in \X \,,
\end{equation}
and  the laws of $N(X^\Theta)$ conditioned to be in $\X^+$, 
under $\PP$ and under $\tPP$ are equal. That is,
\begin{equation}
\label{doobx1}
\PP_0(X^\Theta=x \, | \, X^\Theta\in \X^+)=\tPP_0(X^\Theta=x \, | \, X^\Theta\in \X^+)\,,
x\in \X^+.
\end{equation}
Finally,
\begin{equation}
\label{dooby1}
\PP_0(N(X^\Theta)=x \, | \, X^\Theta\in \X^+)
=\tPP_0(N(X^\Theta)=x | X^\Theta\in \X^+ )\,, x\in \X^+.
\end{equation}
\end{proposition}

\begin{proof}
Since $\tPP_0(\theta<\infty)=1$, then $X^\Theta$  
has only finite excursions $\tPP-$a.s.,  
and from (\ref{doobx}) one gets that the law of $X^\Theta$ 
under $\tPP$, is equal to the law of $X^\Theta$ under $\PP_0$ conditioned to 
$\Theta$ be finite. 

\medskip

Let us prove (\ref{doobx1}). From (\ref{descx}), (\ref{beta0}) and 
(\ref{wprobtr}) we get
$$
\tPP_0(X^\Theta\in \X^+)\PP_0(\tau^+<\infty)
=\tp_0 \beta_0=p_0 \beta_1=\PP_0(X^\Theta\in \X^+).
$$
Then, for  $x=(i_0,...,i_n)\in \X^+$ we obtain,
\begin{eqnarray*}
\tPP_0(X^\Theta=x | X^\Theta\in \X^+)&=&
\left(\prod_{i=0}^{n-1} p(i_{k-1},i_k)\right)
{\left(\tPP_0(X^\Theta\in \X^+)\PP_0(\tau^+<\infty)\right)}^{-1}\\
&=&\left(\prod_{i=0}^{n-1} p(i_{k-1},i_k)\right) \PP_0(X^\Theta\in \X^+)^{-1}\\
&=&\PP_0(X^\Theta=x |X^\Theta\in \X^+) .
\end{eqnarray*}
Then,  (\ref{doobx1}) holds.  So, also (\ref{dooby1}) is satisfied 
because (\ref{doobx1})  implies that 
the laws of the level counting processes  $N(X^\Theta)$ 
conditioned to $\{X^\Theta\in \X^+\}$, 
under $\tPP$ and $\PP$, are equal. 
\end{proof}

\medskip

Let us see what happens in the homogeneous case.  If $p=q$, then
the $Q-$chain is the same as the original one because $\beta_i=1$
for all $i\in \ZZ$. Let us assume $p<q$, then
$\beta_i=1$ and $\beta_{-i}=(p/q)^i$ for $i\ge 1$. For $i=0$ one has
$$
\beta_0=p+q\frac{p}{q}=2p.
$$
Hence, $\beta_{i+1}/\beta_i=1=\beta_{i-1}/\beta_i$, 
$\beta_{-(i+1)}/\beta_{-i}=p/q$ 
and $\beta_{-(i-1)}/\beta_{-i}=q/p$ for $i\ge 2$. 
Hence, the transition probabilities are 
$$
\tp_i=p, \; \tq_i=q \hbox{ and } \tp_{-i}=q, \;\;  \tq_{-i}=p \; \hbox{ for } i\ge 2.   
$$
On the other hand 
$$
\tp_1=p , \; \tq_1=q   \hbox{ and } \tp_{-1}=p\frac{q}{p}=q, \;\;  
tq_{-1}=q\frac{p}{q}=p.
$$
Finally
$$
\tp_0=p\frac{1}{2p}=\frac{1}{2} \hbox{ and } \tq_0=q\frac{p}{2pq}=\frac{1}{2}.
$$
Therefore in the $Q-$matrix we have $1/2$ the probability to start 
a positive or a negative excursion is $1/2$ and the excursions are sign symmetric. 

\medskip

When $p>q$, the analysis is the same as the one just made, 
we must only exchange the behavior of the positive excursions with
the behavior of negative excursions and exchange the roles of $p$ and $q$.
Therefore, for the homogeneous random walk we have 
$\tPP(X^\Theta=x)=\tPP(X^\Theta=-x)$ for all $x\in \X$.
Then, from (\ref{doobz1}) one gets,
\begin{equation}
\label{symmh}
\PP_0(X^\Theta=x \, | \, \Theta<\infty)
=\PP_0(X^\Theta=-x \, | \, \Theta<\infty) ,\,  x\in \X.
\end{equation}
So, under $\PP$ conditioned to $\Theta$ be finite, the distribution 
of all the quantities are sign symmetric. Thus,  
the distributions of $H(X^\Theta)$ and 
$-H(-X^\Theta)$ conditioned to $\{\Theta<\infty\}$ are the same, as well 
the distributions of $(N_h(X^\Theta): h\in \ZZ)$ 
and $(N_{-h}(-X^\Theta): h\in \ZZ)$ conditioned to $\{\Theta<\infty\}$.  

\medskip

In the homogeneous case when $p\le q$, 
$N(X^\Theta)$ is the linear fractional Galton Watson process of parameter $p$.
Since $\{X^\Theta\in \X^{+,U}\}=\{N_{H(X^\Theta)-1}=1\}$, 
then  $N(X^\Theta)$ conditioned to $\{X^\Theta\in \X^{+,U}\}$ 
is the linear fractional Galton Watson process
conditioned to $\{N_{H(X^\Theta)-1}=1\}$.  From the symmetry (\ref{symmh})
and  relation (\ref{vv3}) we get that the Vervaat transformed level counting 
process $N(\V(X^\Theta))$ conditioned to  $\{X^\Theta\in \X^{+,U}\}$ is
the time reversed linear fractional Galton Watson process of parameter $p$,
$$
(N_{H(X ^\Theta)-1-h}(X^\Theta)=0,..,H(X^\Theta)-1),
$$ 
conditioned to $\{N_{H(X^\Theta)-1}=1\}$.

\bigskip

\section*{Acknowledgments}
This work was supported by the Center for Mathematical Modeling, ANID 
Basal Project FB210005.


\bigskip

\begin{thebibliography}{99}


\bibitem{AN}  Athreya, K. B. and Ney, P. (1972) \textit{Branching Processes.}
Springer, New York.

\bibitem{ch} Champagnat, N. (2015) Processus de Gaton-Watson et applications 
en dynamique des polulations. Master, \'Ecole Sup\'erieure des Sciences 
et Technologies  de Hammam Sousse, Tunisie.p.  46  https://hal.inria.fr/cel.-01216832.

\bibitem{Harris} (1952) First passage and recurrence distributions. Trans. Amer. 
Math. Soc. Vol. 73, No. 3 pp. 471-486.

\bibitem{Harrisb}  Harris, T. E. (1963) \textit{The theory of branching
processes.} Die Grundlehren der Mathematischen Wissenschaften, Bd. 119
Springer-Verlag, Berlin; Prentice-Hall, Inc., Englewood Cliffs, N.J.

\bibitem{Kleb} Klebaner, F. C., R\"{o}sler, U. and Sagitov, S. (2007)
Transformations of Galton-Watson processes and linear fractional  
reproduction. \textit{Advances in Applied Probability,} Vol. 39, No. 4,
1036-1053.

\bibitem{lpt} Lupu, T., Pitman J., Tang W. (2015). Electron. J. Probab. 20, Vol. 20, 
No. 51, 1-31.

\bibitem {mar} Marckert, J-F. (1999)  Marches Al\'eatoires, arbres et optimalit\'e 
d'algoritmes. PhD. Thesis, Universit\'e Henri Poincar\'e Nancy I. 

\bibitem{Neveu} Neveu, J. (2006) Arbres et processus de Galton-Watson. 
Annales de l'I.H.P., section B, tome 22, No. 2, pp. 199-207.

\bibitem{Pit} Pitman, J. (1997) Enumerations of trees and forest related to branching 
processes and random walks. Technical Report No. 482, University of California. 

\bibitem{Ve} Vervaat W.  (1979) A relation between Brownian bridge and Brownian 
excursion. Ann. Probab. 7 (1) 143-149. 

\end{thebibliography}
\end{document}